\documentclass{elsarticle}

\usepackage{amssymb, latexsym, comment, amsmath, amsthm, upgreek, epsf, epsfig, color, verbatim, caption, tabularx, xfrac, float}
\usepackage{amsfonts}
\usepackage{booktabs}

\usepackage{footnote}
\usepackage{url}
\usepackage[paperheight=11in,paperwidth=8.5in,margin=1.25in]{geometry}

\usepackage{array}
\newcolumntype{L}[1]{>{\raggedright\let\newline\\\arraybackslash\hspace{0pt}}m{#1}}
\newcolumntype{C}[1]{>{\centering\let\newline\\\arraybackslash\hspace{0pt}}m{#1}}
\newcolumntype{R}[1]{>{\raggedleft\let\newline\\\arraybackslash\hspace{0pt}}m{#1}}

\makeatletter
\def\ps@pprintTitle{%
\let\@oddhead\@empty
\let\@evenhead\@empty
\def\@oddfoot{\centerline{\thepage}}%
\let\@evenfoot\@oddfoot}
\makeatother

\newtheorem{theorem}{Theorem}[section]
\newtheorem{lemma}[theorem]{Lemma}
\newtheorem{proposition}[theorem]{Proposition}

\newtheorem{definition}[theorem]{Definition}

\DeclareMathOperator\arctanh{arctanh}
\DeclareMathOperator\sech{sech}

\begin{document}

\begin{frontmatter}
\title{The length of the shortest closed geodesic on positively curved 2-spheres}
\author{Ian Adelstein and Franco Vargas Pallete}
\address{Department of Mathematics, Yale University \\ New Haven, CT 06520 United States}
\begin{abstract} We show that the shortest closed geodesic on a 2-sphere with non-negative curvature has length bounded above by three times the diameter. We prove a new isoperimetric inequality for 2-spheres with pinched curvature; this allows us to improve our bound on the length of the shortest closed geodesic in the pinched curvature setting. 
\end{abstract}
\begin{keyword} closed geodesics; geodesic nets; spheres
\MSC[2010]  53C22 \sep 53C23
\end{keyword}
\end{frontmatter}

\section{Introduction}

Gromov \cite{Gro83} has asked if there exist constants $c(n)$ such that the length of the shortest closed geodesic $L(M^n)$ on a closed Riemannian manifold $M^n$ is bounded above by $c(n) D(M^n)$, where $D(M^n)$ is the diameter of the manifold. On non-simply connected manifolds the shortest non-contractible closed curve is a geodesic with length bounded above by $2D(M^n)$. On manifolds homeomorphic to the 2-sphere Croke \cite{croke1} provided the first bound of $L(S^2,g) \leq 9D(S^2,g)$, which was improved by Maeda \cite{maeda}, and finally by Nabutovsky and Rotman \cite{NR2002} and independently Sabourau \cite{Sab} to $L(S^2,g) \leq 4D(S^2,g)$. Rotman \cite{rot5} has further proved that $L(S^2,g) \leq 4R(S^2,g)$ where $$R(S^2,g)= \min_{x \in S^2} \max_{y \in S^2} d(x,y) \leq D(S^2,g) =  \max_{x \in S^2} \max_{y \in S^2} d(x,y)$$ is the radius of the Riemannian sphere. 
Gromov's question is open for simply connected manifolds in dimensions $n \geq 3$. 

An attractive conjecture is that $L(M^n) \leq 2D(M^n)$ for all closed Riemannian manifolds $M^n$. To explore this bound one might consider Zoll spheres:~metrics on the 2-sphere all of whose geodesics are closed and of the same length. The conjecture turns out to be overly optimistic, as Balacheff, Croke, and Katz \cite{bck} have produced Zoll spheres with $L(S^2, Zoll) > 2D(S^2, Zoll)$. These examples are not constructive, and it is unknown how much longer than $2D(S^2, g)$ the shortest closed geodesic could be. In this paper we prove the following:

\begin{theorem}\label{main} Non-negatively curved 2-spheres have $L(S^2, g) \leq 3R(S^2, g) \leq 3D(S^2, g)$.
\end{theorem}

While one does not expect the inequality $L(S^2, g) \leq 3D(S^2, g)$ to be sharp, we note that the inequality $L(S^2, g) \leq 3R(S^2, g)$ is realized by the metric space formed when gluing two equilateral triangles along their common boundaries, the so called Calabi-Croke sphere. 
The centers of the triangles realize the radius whereas the vertices realize the diameter. The shortest closed geodesic is a doubled altitude which has length exactly 3 times the radius and $\sqrt{3} $ times the diameter. 

We should note that the results cited above are curvature free bounds, whereas our bounds require non-negative curvature. Calabi and Cao \cite{calabi} studied simple closed geodesic on non-negatively curved 2-spheres, and showed that any closed geodesic of shortest length must be simple. We can therefore improve our Theorem~\ref{main} to the following:~on non-negatively curved $(S^2,g)$ the shortest closed geodesic is simple and has length bounded above by three times the radius.

Additional work in the positive curvature setting includes \cite{recent} where pinched metrics on the 2-sphere are studied and an upper bound for the length of the shortest closed geodesic is provided in terms of the area. We combine this result with a new isoperimetric inequality for pinched metrics on the 2-sphere (Theorem~\ref{iso}) to yield the following:

\begin{theorem}\label{pinched}
Let $(S^2,g)$ be a $\delta$-pinched metric with $\delta > \frac{4+\sqrt{7}}{8} \approx 0.83$. Then

\[L(S^2,g) \leq  \frac{2}{\sqrt{\delta}}D(S^2,g)
\]

with equality if and only if the sphere is round.

\end{theorem}

For $\delta = .83$ our theorem yields the bound $L(S^2,g) \leq 2.19 D(S^2,g)$. The theorem is optimal in the sense that the constant $\frac{2}{\sqrt{\delta}}$ converges to $2$ as $\delta$ approaches $1$. Finally, we note that this theorem can be used to further comment on the Zoll spheres due to Balacheff, Croke, and Katz \cite{bck} where we can now say that $$2D(S^2,Zoll) < L(S^2, Zoll) < \frac{2}{\sqrt{\delta}}D(S^2,Zoll).$$

The paper proceeds as follows. In Section~\ref{rot} we present a proof due to Rotman \cite{video} of the fact that $L(S^2,g) \leq 4D(S^2,g)$. A crucial step in the proof uses a weighted length shortening to avoid stationary theta-graphs (critical points of the length functional on nets). This weighted flow increases the bound on the shortest closed geodesic from $3D(S^2,g)$ to $4D(S^2,g)$. In Section~\ref{deets} we show for positive curvature metrics on the 2-sphere that non-trivial stationary theta-graphs are never local minima. The weighted flow in Rotman's proof can therefore be avoided, and Theorem~\ref{main} follows. In Section~\ref{pinch} we prove a new isoperimetric inequality for pinched metrics on the 2-sphere and combine this with the main result of \cite{recent} to prove Theorem~\ref{pinched}. The new isoperimetric inequality is further refined in the Appendix by studying a Sturm-Liouville problem (SLP) related to lower bounds on the first eigenvalue of the Laplace-Beltrami operator.

\textbf{Acknowledgements:}~The authors would like to recognize Christina Sormani and her NSF grant DMS-1612049. The grant made possible a conference at Yale where the authors learned of many of the results cited in this paper and had productive conversations with Alexander Nabutovsky, Regina Rotman, and Frank Morgan. The authors would also like to thank Stéphane Sabourau for helpful suggestions, as well as Ben Andrews for his helpful advice on how to linearly approximate eigenvalues of the related SLP. We also acknowledge Wolfgang Ziller for helpful discussions about short closed geodesics in the pinched curvature setting, the fruits of which will appear in a forthcoming paper. Finally, we want to acknowledge the generous contribution of the referee who indicated how our proof of the inequality $L(S^2,g) \leq 3D(S^2,g)$ could be extended to the sharp bound $L(S^2,g) \leq 3R(S^2,g)$.

\section{Proof that $L(S^2,g) \leq 4D(S^2,g)$}\label{rot}

We start with some preliminaries that will be used both in the proof of the bound $L(S^2,g) \leq 4D(S^2,g)$ and throughout the remainder of the paper. 

\begin{definition} 
A \emph{geodesic net} is a finite graph immersed in a Riemannian manifold such that each edge is a geodesic segment. A geodesic net is said to be \emph{stationary} if at each vertex the sum of the unit vectors tangent to the incident edges equals zero.
\end{definition}

As such, stationary geodesic nets are critical points of the length functional on the space of nets. Closed geodesics are the first examples of stationary geodesic nets. A figure eight curve is a stationary geodesic net if each loop is geodesic and if the stationarity condition is satisfied at the vertex. In dimension two the stationarity condition implies that a geodesic net based on the figure eight curve will be a self-intersecting closed geodesic. 

An example of a stationary geodesic net that is not a closed geodesic is the stationary theta-graph. A theta-graph is a net consisting of exactly two vertices joined by exactly three edges. The stationarity condition ensures that these edges pairwise meet at angle $\frac{2\pi}{3}$ at each vertex. Hass and Morgan \cite{morgan} gave one of the only known existence results for geodesic nets, demonstrating that convex metrics on the 2-sphere nearby the round metric admit stationary theta-graphs. 

Nabutovsky and Rotman \cite{NR2002} and independently Sabourau \cite{Sab} gave the original proofs that $L(S^2,g) \leq 4D(S^2,g)$. In working to improve this bound Rotman obtained alternate unpublished proofs, one of which we present here \cite{NR2005, rot2, video, rot3}. This proof uses a pseudo-filling technique, analogous to the technique introduced by Gromov \cite{Gro83} in proving bounds for essential manifolds on the length of the shortest closed geodesic in terms of volume.

\begin{theorem}[\cite{NR2002, Sab}]\label{4d} Riemannian 2-spheres have $L(S^2,g) \leq 4D(S^2,g)$. 
\end{theorem}

\begin{proof}[Proof \cite{video}.]
Let $M=(S^2,g)$ be a Riemannian 2-sphere and $f \colon (S^2, std) \to M$ a diffeomorphism. We attempt to extend $f$ to a map $\tilde{f} \colon (D^3, std) \to M$. As $M$ is a 2-sphere such a map should not exist, and as an obstruction to this extension we obtain a periodic geodesic on $M$ with length $\leq 4D(M)$. 

First triangulate $(S^2, std)$ such that the diameter of the triangulation on $M$ induced by $f$ is less than $\delta$. Next triangulate $(D^3, std)$ as a cone over the triangulated $(S^2, std)$, i.e.~add a single vertex $p \in D^3$ at the center of the ball and the corresponding 1, 2, and 3-simplexes. We attempt to extend the map $f$ inductive to this skeleton.

{\bf 0-skeleton}:~We need only choose a point $\tilde{p} \in M$ with $\tilde{f}(p) = \tilde{p}$. 

{\bf 1-skeleton}:~Let $v_i $ be the vertices of the triangulation of $S^2$ and $f(v_i)=\tilde{v}_i$ the corresponding vertices of the induced triangulation of $M$. We send the 1-simplex between $p$ and $v_i$ on $D^3$ to a minimizing geodesic between $\tilde{p}$ and $\tilde{v}_i$ on $M$ with length less than the diameter of $M$. 

{\bf 2-skeleton}:~We attempt to send the 2-simplex on $D^3$ associated to the triple $(p, v_i, v_j)$ to a 2-simplex on $M$ associated to the triple $(\tilde{p}, \tilde{v}_i, \tilde{v}_j)$. The triple on $M$ is already connected by 1-simplexes that form a piecewise smooth closed curve with length less than $2D(M)+\delta$. We use Birkhoff curve shortening process to deform this closed curve without increasing its length, either to a closed geodesic with length less than $2D(M)+\delta$, or to a point in which case we have swept out the desired 2-simplex.

{\bf 3-skeleton}:~We attempt to send the 3-simplex on $D^3$ associated to the tuple $(p, v_i, v_j, v_k)$ to a 3-simplex on $M$ associated to the tuple $(\tilde{p}, \tilde{v}_i, \tilde{v}_j , \tilde{v}_k)$. By the previous steps we know where the boundary of this 3-simplex is sent; call this boundary 2-sphere $S^2_0 \subset M$. If we are able to contract $S^2_0$ to a point, i.e.~construct a homotopy $S^2_t$ with $S^2_1 = \{x\}$, then we will have succeeded in sending 3-complexes to 3-complexes, thus extending the map $f$ to $\tilde{f}$. Such an extension is not possible, and as an obstruction we obtain a short periodic geodesic on $M$. 

We first contract the small 2-simplex associated to the triple $(\tilde{v}_i, \tilde{v}_j , \tilde{v}_k)$ to a point which we call $\tilde{v} \in M$ (c.f.~\cite{rot3}, Remark ending Section 1). We then have a theta-graph between the pair of points $(\tilde{p}, \tilde{v})$ consisting of three 1-simplexes, which we call $e_1, e_2, e_3$. As before, the Birkhoff curve shortening process on each pair of 1-simplexes $\{ e_i , e_j \}$ yields the boundary 2-sphere $S^2_0$. 

In order to construct the homotopy $S^2_t$ we first use length shortening flow for nets to deform the theta-graph to a point, c.f.~\cite[Section 3]{NR2005}. At each time in this deformation we apply the Birkhoff curve shortening process to each pair of edges, sweeping out the desired $S^2$. The continuity of the Birkhoff curve shortening process (with respect to the initial pair $\{ e_i , e_j \}$ of edges) in the absence of short closed geodesics is what allows us to extend the homotopy which contracts the theta-graph to the desired homotopy $S^2_t$.

We therefore need only study the situation in which the theta-graph gets stuck on a stationary geodesic net before contracting to a point during the length shortening process. There are three cases to consider:

Case 1: The theta-graph degenerates to a periodic geodesic; this geodesic will have length less than $3D(M)$.

Case 2: One of the edges disappears during the length shortening yielding a stationary figure eight with length less than $3D(M)$. The stationarity condition in dimension two implies that this is a (self-intersecting) periodic geodesic. 

Case 3: The theta-graph gets stuck on a stationary theta-graph. In this situation we apply a weighted length shortening process. Let $(w_1, w_2, w_3)$ be the triple of unit direction vectors at a vertex of a theta-graph. We consider a weighted length shortening flow where we double the weight of the third vector. The stationarity condition is then $| w_1 + w_2 + 2w_3 | = 0$ which implies that stationary theta-graphs are not critical points of the the weighted flow. Critical points occur when $w_1$ and $w_2$ collapse to a single edge or one of these edges disappears, which means we are in one of the two previous cases. Because we doubled the weight of one of the edges we now produce a (potentially self-intersecting) periodic geodesic with length bounded above by $4D(M)$. 
\end{proof}

\section{Positive metrics on the 2-sphere} \label{deets}

In this section we indicate how the proof of Theorem~\ref{4d} adapts in the positive curvature setting to yield our Theorem~\ref{main}. Under the positive curvature assumption we show that stationary theta-graphs are never local minima of the length functional on nets, allowing us to avoid weighted length shortening. Once we prove the theorem in the positive setting, we show how it extends to the non-negative setting by considering conformally close positive metrics. 
We begin by recalling the first and second variations of length, which can be found for instance in \cite[Section 5.1]{Jost}, see also \cite{morgan2} and \cite{morgan3} for formula that apply more directly in the setting of stationary nets.

\begin{proposition}[First variation of length, Lemma 5.1.1 \cite{Jost}]\label{1stvar} Given a smooth curve $\gamma:[a,b]\rightarrow M$ parametrized by arc-length and a vector field $V$ on $\gamma$, let $H$ be a variation of $\gamma$ in the direction of $V$ so that $H:[a,b]_t\times[-\epsilon,\epsilon]_s \rightarrow M$ is smooth, $H(\cdot,0)=\gamma,\, \frac{dH}{ds}|_{s=0}(\cdot,0)=V$. If we denote by $L(s):= \ell(H(\cdot,s))$ then
\[L'(0) = \langle V, \gamma'\rangle|_a^b - \int_a^b \langle V(t),\nabla_{\gamma'}\gamma'(t) \rangle dt
\]

\end{proposition}

\begin{proposition}[Second variation of length, Theorem 5.1.1 \cite{Jost}]\label{2ndvar} Given a smooth geodesic $\gamma:[a,b]\rightarrow M$ parametrized by arc-length and a vector field $V$ on $\gamma$, let $H$ be a variation of $\gamma$ in the direction of $V$ so that $H:[a,b]_t\times[-\epsilon,\epsilon]_s \rightarrow M$ is smooth, $H(\cdot,0)=\gamma,\, \frac{d}{ds}|_{s=0}H(\cdot,0)=V$. If we denote by $V^\perp$ the perpendicular projection of $V$ with respect to $\gamma'$ and by $L(s):= \ell(H(\cdot,s))$ then
\[L''(0) = \biggl\langle\frac{D}{ds}\frac{dH}{ds}, \gamma'\biggr\rangle\bigg|_{(a,0)}^{(b,0)} + \int_a^b \Vert \nabla_{\gamma'}V^\perp(t)\Vert^2 - \langle R(V^\perp(t),\gamma'(t))\gamma'(t), V^\perp(t) \rangle dt
\]

\end{proposition}

\begin{lemma}\label{lem} Any stationary theta-graph on a positively curved 2-sphere admits directions of decrease (within the space of nets) for the length shortening flow.
\end{lemma}

\begin{proof}
We simply demonstrate a variation with negative second variation of length. Give each edge a unit speed parametrization $\gamma_i \colon [a_i,b_i] \to (S^2,g)$ and define  vector fields $V_i$ so that $V_1^\perp$ and $V_2^\perp$ are of constant size $1$ (hence parallel), $V_3^\perp\equiv0$, and the $V_i$ all agree at the vertices of the theta-graph. For example: \\

$V_1(t) = \frac{1}{\sqrt{3}}\cos{(\frac{t-a_1}{b_1-a_1}\pi)}\dot{\gamma}_1+1\dot{\gamma}_1^{\perp}$ \\

$V_2(t) = \frac{1}{\sqrt{3}}\cos{(\frac{t-a_2}{b_2-a_2}\pi)}\dot{\gamma}_2-1\dot{\gamma}_2^{\perp}$ \\

$V_3(t) = \frac{-2}{\sqrt{3}}\cos{(\frac{t-a_3}{b_3-a_3}\pi)}\dot{\gamma}_3+0\dot{\gamma}_3^{\perp}$ \\

For the given variational fields $V_i$, we choose variations $H_i(\cdot,s)$ which agree at the vertices; for example, one could set $H_i(t,s)=\exp_{\gamma_i(t)}sV_i(t)$. The fact that the $H_i(\cdot,s)$ agree at the vertices ensures that we are deforming through theta-graphs. Moreover, as the variations keep each edge embedded, and maintain angles close to the initial $\frac{2\pi}{3}$ angles, we are guaranteed the theta-graphs remain embedded during the deformation.

If we denote by $L(s)$ the sum of lengths of $H_1(\cdot,s), H_2(\cdot,s), H_3(\cdot,s)$, then by the first variation formula (Proposition \ref{1stvar}) we have that

\begin{equation}\label{eq:length} L'(0)= \sum_{i=1}^3\langle V_i, \gamma_i'\rangle|_{a_i}^{b_i},
\end{equation}\noindent
since $\lbrace\gamma_i\rbrace_{i=1}^3$ are geodesics. And because the vector fields $\lbrace V_i\rbrace_{i=1}^3$ agree at the endpoints and the geodesics meet at angles $\frac{2\pi}{3}$, the summands in Equation~\ref{eq:length} cancel out for each vertex. Hence $L'(0)=0$.

For the second variation, note that because $\lbrace V^\perp_i\rbrace_{i=1}^3$ are parallel and $M$ has positive curvature we have that for $1\leq i\leq3$

\[\int_{a_i}^{b_i} \Vert \nabla_{\gamma_i'}V_i^\perp(t)\Vert^2 - \langle R(V_i^\perp(t),\gamma_i'(t))\gamma_i'(t), V^\perp_i(t) \rangle dt < 0
\]

Applying the second variation formula (Proposition \ref{2ndvar}) we have then

\begin{equation}\label{eq:2ndvar}
L''(0)< \sum_{i=1}^3\biggl\langle\frac{D}{ds}\frac{dH_i}{ds}, \gamma_i'\biggr\rangle\bigg|_{(a_i,0)}^{(b_i,0)}
\end{equation}

Given our choice of $H_i(t,s)=\exp_{\gamma_i(t)}sV_i(t)$ we see that $\frac{D}{ds}\frac{dH_i}{ds}=0$ for each $i \in \{1,2,3\}$ and therefore that the right hand side of Equation~\ref{eq:2ndvar} vanishes. Thus under the described variation we have $L''(0)<0$, which together with $L'(0)=0$, implies that length decreases for small values of $s$ in $\lbrace H_i(\cdot,s)\rbrace_{i=1}^3$.
\end{proof}

We show by example that Lemma~\ref{lem} is sharp in the sense that there exist stationary theta-graphs on non-negatively curved 2-spheres which do not admit directions of decrease. Consider the metric space formed by gluing two equilateral triangles along their common boundaries, so that a geodesic on the top face billiards around an edge to the bottom face. This \emph{doubled triangle} is a 2-sphere with flat metric and three conical singularities; it is sometimes called the Calabi-Croke sphere \cite{Sab2}.

The doubled triangle admits a degenerate stationary theta-graph. Connect the center of the top face with the center of the bottom face via three geodesic segments which pass perpendicularly through each edge of the triangular boundary (see Figure~\ref{graph}). Such an arrangement ensures that the edges of the graph meet at the vertices at angle $2\pi /3$. By moving the vertices of the graph towards one of the vertices of the triangle, and keeping the edges of the graph perpendicular to the edges of the triangle, one produces a degenerate family of stationary theta-graphs, all having the same total length, and failing to admit directions of decrease. 

\begin{figure}[h]
  \includegraphics[scale=.25]{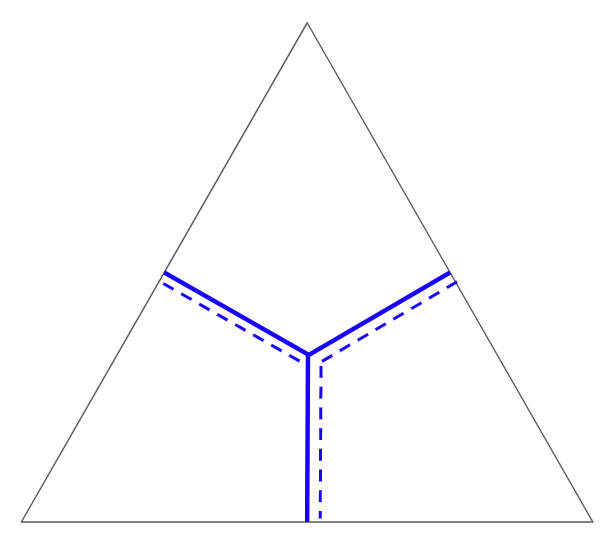}
  \centering
  \caption{A degenerate theta-graph on the doubled triangle, where edges on the top face are solid and edges on the bottom face are dashed.}
  \label{graph}
\end{figure}

Because these theta-graphs on the doubled triangle avoid the vertices of the triangle, this degenerate family also exists on smooth non-negative 2-sphere metrics close to the doubled triangle metric. This example illustrates that the proof of Theorem~\ref{main} in the non-negative setting can not rely directly on Lemma~\ref{lem}, and instead will follow by first proving the theorem in the positive setting, and then extending by considering conformally close positive metrics.

\begin{proof}[Proof of Theorem~\ref{main}]

Let us assume first $M = (S^2, g)$ to be a metric of positive curvature. We follow the proof of Theorem~\ref{4d}, trying to extend the map $f \colon (S^2, std) \to M$ to a map $\tilde{f} \colon (D^3, std) \to M$. When extending the 0-skeleton we now choose a point $\tilde{p} \in M$ that realizes the radius, i.e.~such that $B(\tilde{p},R)$ covers $M$. This choice ensure that $\tilde{p}$ is at distance at most $R(M)$ from the points in the triangulation; it is the same choice that Rotman \cite{rot5} makes when improving her bound from $4D(S^2, g)$ to $4R(S^2, g)$. 

Next we follow the proof of Theorem~\ref{4d} until Case 3 where the theta-graph gets stuck on a stationary theta-graph, a critical point of the length shortening flow for nets. By Lemma~\ref{lem} this critical point admits a direction of decrease, allowing us to continue the contraction past the stationary theta-graph, and eliminating the necessity of the weighted length shortening. 

While we are required to make a choice about the direction we deform from a stationary theta-graph, this choice can be made independently in each $3$-cell; we do not expect this extended (restarted) flow to depend continuously on the initial theta-graph. Indeed, we only need the fact \cite[Lemma 4]{NR2005} that the space of theta-graphs with length less than $3R(M)+\delta$ is connected in order to contract the initial theta-graph to a point. The continuity of the Birkoff curve shortening process (with respect to an initial pair of edges) in the absence of short closed geodesics allows us to extend this contraction to the entire 3-cell.

.

Upon resuming the flow, it is possible that we encounter another stationary theta-graph. Note that this stationary theta-graph will be distinct from the first, as the length shortening flow is strictly decreasing (preventing us from visiting the same sequence of theta-graphs repeatedly). It is possible running the flow in this extended (restarted) manner that we encounter a sequence of stationary theta-graphs accumulating to some limit object. In this case an application of transfinite induction ensures that this limit object is again a stationary theta-graph from which we can continue the flow.

As it is impossible to contract all theta-graphs to a point (this would extend the map $f \colon (S^2, std) \to M$ to a map $\tilde{f} \colon (D^3, std) \to M$) we must end in Case 1 or 2. By avoiding the weighted length shortening  we produce a (potentially self-intersecting) closed geodesic with length bounded above by $3R(M) \leq 3D(M)$.

The theorem is thus proved in the positive case. In the case where the sphere is non-negatively curved, we proceed as follows. Choose a smooth function $\varphi \colon S^2 \to \mathbb{R}$ such that $\Delta\varphi < 0$ on the set $K_{0} = \lbrace x\in S^2 \,|\, K_x=0 \rbrace$. The existence of such a $\varphi$ is possible because $K_0$ is a proper subset of the sphere (via Gauss-Bonnet). We consider the metrics $g_t = e^{2t\varphi}g$ which have strictly positive curvature $K_t = e^{-2t\varphi}(K-t\Delta\varphi)$ for $t>0$ small. Applying the result for positively curved metrics to $(S^2,g_t)$, and letting $t\rightarrow 0^+$, we use that the convergence $(S^2,g_t)\rightarrow(S^2,g)$ is smooth  and that $L(S^2,.)$ is lower-semicontinuous (this last property follows from the fact that in smooth convergence closed geodesics converge to closed geodesics) to conclude that the inequality $L(S^2,g) \leq 3R(S^2,g)$ holds in the non-negative setting.
\end{proof}

Finally, note that in \cite{calabi} it is proved that the shortest closed geodesic on a non-negatively curved $(S^2,g)$ is simple. Therefore, if the obstruction in the above proof yields a self-intersecting periodic geodesic (Case 2), then there also exists a simple closed geodesic with length bounded above by three times the radius. In short, we can use the main result of \cite{calabi} to improve our Theorem~\ref{main} to the following: on non-negatively curved $(S^2,g)$ the shortest closed geodesic is simple and has length bounded above by three times the radius.

\section{Pinched metrics on the 2-sphere}\label{pinch}

The main goal of this section is to prove the following new isoperimetric inequality for pinched metrics on the 2-sphere. A positive metric $(S^2,g)$ is said to be $\delta$-pinched if $K_{\min}/K_{\max}\geq \delta$.

\begin{theorem}\label{iso}
Let $(S^2,g)$ be a $\delta$-pinched metric. Then $$\pi A \leq \frac{4D^2}{\delta}$$ with equality if and only if the sphere is round. 
\end{theorem}

We combine this inequality with the main result from \cite{recent} in order to prove Theorem~\ref{pinched}. The main result of \cite{recent} is achieved via a combination of techniques from Riemannian and symplectic geometry.

\begin{theorem}[\cite{recent}]\label{thm:L2Aineq}
Let $(S^2,g)$ be a $\delta$-pinched metric with $\delta > \frac{4+\sqrt{7}}{8} \approx 0.83$. Then $L^2(S^2, g) \leq \pi A(S^2, g)$ where $A(S^2,g)$ is the area.
\end{theorem}

By combining Theorem~\ref{iso} with Theorem~\ref{thm:L2Aineq}, we have for $(S^2,g)$ a $\delta$-pinched metric with $\delta > \frac{4+\sqrt{7}}{8} \approx 0.83$, that  $$L^2(S^2, g) \leq \pi A(S^2, g) \leq   \frac{4D^2(S^2,g)}{\delta} $$ and therefore that $$L(S^2,g) \leq  \frac{2}{\sqrt{\delta}}D(S^2,g).$$ 

Equality here implies equality in the isoperimetric inequality, and therefore that the sphere is round by Theorem~\ref{iso}. We have thus proved Theorem~\ref{pinched} and all that remains in this section is the proof of Theorem~\ref{iso}.

We first note that Theorem~\ref{iso} is a curvature pinched version of the following result due to Calabi and Cao. Moreover, the proof techniques we use are adaptations of theirs to the pinched curvature setting.

\begin{theorem}[\cite{calabi}, Theorem C]
Let $(S^2,g)$ have non-negative curvature. Then $A \leq \frac{8}{\pi} D^2$.
\end{theorem}

Calabi and Cao proved the above inequality by combining an upper bound on the first eigenvalue $\lambda_1(S^2,g) \leq \frac{8\pi}{A(S^2,g)}$ due to Hirsch  \cite{hersch, YangYau} with a lower bound $ \frac{\pi^2}{D^2(M^n,g)} \leq \lambda_1 (M^n,g)$ due to Zhong and Yang \cite{zhongyang} which holds in the non-negative Ricci setting. This lower bound on $\lambda_1(M^n,g)$ has been improved many times in the setting of positive lower bound on Ricci (see \cite{he} for a survey or \cite{AndrewsClutterbuck} for the optimal bound). We use the version from \cite{AndrewsClutterbuck}, which we state for $\lambda_1(S^2,g)$.

\begin{theorem}[\cite{AndrewsClutterbuck}]\label{thm:evlowerbound}
The quantity
\[\lambda_1(d,k) = \inf\lbrace \lambda_1(S^2,g)\,|\, Diam(S^2,g)\leq d,\,K(S^2,g)\geq k\rbrace
\]
is equal to the first eigenvalue $\mu$ of the following Sturm-Liouville problem (SLP) with Neumann initial conditions
\[ (y'\cos(\sqrt{k}x))' + \mu\cos(\sqrt{k}x)y = 0,\quad y'(\pm d/2)=0
\]
\end{theorem}

We can apply this directly to the case $d=\pi,\, k=1$ to obtain the case ${\rm dim}=2$ of the classical result of Lichnerowicz \cite{Lichnerowicz}.

\begin{lemma}\label{two} For $(S^2,g)$ with $Diam(S^2,g)\leq \pi$ and $\,K(S^2,g)\geq 1$ we have $\lambda_1 \geq 2.$
\end{lemma}
\begin{proof}
The first eigenfunction of $(y'\cos(x))' + \mu\cos(x)y = 0,~ y'(\pm \pi/2)=0$ is $y(x)=\sin(x)$ which has eigenvalue $2$.
\end{proof}

We can now provide a proof of Theorem~\ref{iso}. The main insight is that when considering diameter bounds, we can use Klingenberg's injectivity radius estimate to translate between the positive and the pinched curvature settings. 

\begin{proof}[Proof of Theorem~\ref{iso}]

For simplicity we will drop $(S^2,g)$ from our notation. Since the inequality is scale invariant, let us rescale the metric so $K_{min}=1$ and by Myers theorem $D \leq \pi$. Lemma~\ref{two} then gives the bound $ \lambda_1\geq2$. Following the ideas of \cite{calabi} we combine this lower bound on $\lambda_1$ with the upper bound due to Hirsch $\lambda_1 \leq \frac{8\pi}{A}$ (see \cite{hersch}, \cite{YangYau}) to yield $A \leq 4\pi$. 
Klingenberg's injectivity radius estimate for positive metrics on 2-spheres says that $$D \geq \frac{\pi}{\sqrt{K_{max}}} = \pi \sqrt{\delta}.$$ We have therefore related both area and diameter to $\pi$, and conclude that $$\pi A \leq 4\pi^2 \leq \frac{4D^2}{\delta}.$$

In the case of equality we have $4\pi^2 = \frac{4D^2}{\delta}$ and therefore $D=\frac{\pi}{\sqrt{K_{max}}}$. This is the limiting case of Klingenberg's injectivity radius estimate, and therefore diameter equals injectivity radius. This is the so called Blaschke condition, and for Blaschke metrics on the 2-sphere Green \cite{green} gives an elementary proof using classical surface geometry that the sphere must be round. 
\end{proof}

We finish by observing that the isoperimetric inequality in Theorem~\ref{iso} is sharp in the sense that the estimate $A/D^2$ approaches $4/\pi$ as we move towards the round metric, i.e.~as $\delta$ approaches $1$. We do not recover the Calabi-Cao inequality as $\delta$ approaches 0, i.e.~for $(S^2,g)$ with non-negative curvature. Note that the Calabi-Cao inequality is not sharp for convex metrics on $(S^2,g)$; Alexandrov has conjectured that the sharp inequality is realized by the doubled disk metric where $A/D^2 = \pi/2$. 

\section{Appendix}

Using more advanced techniques from Sturm-Liouville theory we can improve the isoperimetric inequality in Theorem~\ref{iso} when the sphere is not round. This improved version of the inequality is not needed for the proof of Theorem~\ref{pinched} that was presented in Section~\ref{pinch}.

\begin{theorem}\label{ap_thm}
Let $(S^2,g)$ be a $\delta$-pinched metric and denote by $\eta=D\sqrt{K_{min}}$.  Then 
\begin{equation*}\label{eq:CCpinched}
\pi A \leq \frac{4D^2}{\delta(2-\sin(\eta/2))}.
\end{equation*}
\end{theorem}

First note that $2- \sin(\eta/2) \geq 1$ so that this isoperimetric inequality is indeed an improvement on that of Theorem~\ref{iso}. Moreover, because $2- \sin(\eta/2) > 1$ when $\eta \neq \pi$ we note that Theorem~\ref{ap_thm} together with Cheng's \cite{cheng} rigidity result when $D=\pi/\sqrt{K_{min}}$ provide an alternate proof of the equality case of Theorem~\ref{pinched} that does not depend on the Blaschke ideas from Section~\ref{pinch}.

This new inequality follows as before by combining the inequality due to Hirsch with an improved lower bound on the first eigenvalue. We therefore set out to calculate the linear approximation of $\lambda_1(d,1)$ at $d=\pi$. The main idea is that because the coefficients of the SLP are constant while the domain varies, we can restrict and extend eigenfunctions to compare the values $\lambda_1(d,1)$ as $d\rightarrow\pi$. Let us then define by $\mu(d)$ the SLP eigenvalue $\lambda_1(d,1)$. We then have

\begin{lemma}\label{lem:1sttaylorlambda} For $\pi\geq d >0$ 
\[2 + 2(1-\sin(d/2)) \leq \mu(d).
\]
\end{lemma}
\begin{proof}
Define the change of variable $\tanh(u)=\sin(x)$ on the SLP, helpfully suggested to us by Ben Andrews. This becomes

\begin{equation}\label{eq:modSLP}
    y'' + \mu.\sech^2(u)y = 0,\quad y'(\pm T)=0,\, \quad \tanh(T)=\sin(d/2)
\end{equation}
In particular, for $d=1, T=+\infty$, the eigenvalue $\mu=2$ is realized by $y_{+\infty}=\tanh(u)$.

Recall that the first eigenvalue for a SLP can be written as a Rayleigh quotient

\begin{equation*}
    \mu(d) = \inf_{y\neq 0, y'(\pm T)=0} R[y,T]
\end{equation*}
where

\[R[y,t]= \frac{\int_{-T}^T -y.y''du}{\int_{-T}^T y^2\sech^2(u)du}
\]

Denote by $\hat{y}_T$ the eigenfunction of the SLP (\ref{eq:modSLP}) normalized so $\int_{-T}^T \hat{y}_T^2\sech^2(u)du=1$. Extended $\hat{y}_T$ as a constant to the entire real line so (while keeping the same notation):

\[  \hat{y}_T(u) =
  \begin{cases}
     \hat{y}_T(u) & \text{if $|u|\leq T$} \\
     \hat{y}_T(T) & \text{if $u>T$}\\
     \hat{y}_T(-T) & \text{if $u<-T$}
    
  \end{cases}
\]

In order to use $\mu(1)\leq R[\hat{y}_T,+\infty]$ we will need to bound $\hat{y}_T(\pm T)$. Observe first that given the symmetries of the SLP (\ref{eq:modSLP}) the function $\hat{y}_T$ is an odd function. Since $\mu(d)$ is the first eigenvalue, we know that $\hat{y}'_T$ does not vanish in the open interval $(-T,T)$. Hence $|\hat{y}_T(\pm T)| = \max_{-T\leq u \leq T}  |\hat{y}_T(u)|$, so then

\begin{equation*}
    1=\int_{-T}^T \hat{y}^2_T\sech^2(u)du \leq |\hat{y}_T(\pm T)|^2 \int_{-T}^T \sech^2(u)du \leq 2|\hat{y}_T(\pm T)|^2
\end{equation*}

Replacing now $R[\hat{y}_T,+\infty]$

\begin{equation*}
    2=\mu(\pi)\leq R[\hat{y}_T,+\infty] = \frac{\mu(d)}{1 + 2|\hat{y}_T(T)|^2\int_T^{\infty} \sech^2(u)du}
\end{equation*}

Hence we obtain

\begin{equation*}
    2 + 2(1-\tanh(T)) \leq \mu(d)
\end{equation*}

Replacing now $\tanh(T)=\sin(d/2)$ and using the Taylor series of sine at $\pi/2$ we finally get

\begin{equation*}
    2 + 2(1-\sin(d/2)) \leq \mu(d).
\end{equation*}\end{proof}

We are now ready to prove Theorem~\ref{ap_thm}. 

\begin{proof}
The proof proceeds exactly as the proof for Theorem~\ref{iso} and we begin by rescaling the metric so that $K_{min}=1$ and by Myers theorem $\eta = D \leq \pi$. Lemma~\ref{lem:1sttaylorlambda} then gives the bound $  2 + 2(1-\sin(d/2)) \leq \mu(d) = \lambda_1(d,1)$. Following the ideas of \cite{calabi} we combine this lower bound on $\lambda_1$ with the upper bound due to Hirsch $\lambda_1 \leq \frac{8\pi}{A}$ (see \cite{hersch}, \cite{YangYau}) to yield 

\begin{equation*}
    2 + 2(1-\sin(\eta/2)) \leq \frac{8\pi}{A}
\end{equation*}
which we can manipulate to
\begin{equation*}
    \pi A \leq \frac{8\pi^2}{2 + 2(1-\sin(\eta/2))}.
\end{equation*}

Klingenberg's injectivity radius estimate for positive metrics on 2-spheres says that $$D \geq \frac{\pi}{\sqrt{K_{max}}} = \pi \sqrt{\delta}.$$ 

We have therefore related both area and diameter to $\pi$, and conclude that 

\begin{equation*}\label{eq:sdependance}
    \pi A \leq \frac{8\pi^2}{2 + 2(1-\sin(\eta/2))} = \frac{4\pi^2}{K_{min}(2-\sin(\eta/2))} \leq \frac{4D^2}{\delta(2-\sin(\eta/2))}.
\end{equation*} \end{proof}

Finally, while not necessary for the proof of Theorem \ref{ap_thm}, the following Lemma of independent interest expands the approach of Lemma \ref{lem:1sttaylorlambda} and gives an idea of the sharpness of its inequality, both in general and at the limiting case of $d=\pi$. 

\begin{lemma}
For $\pi\geq d \geq 2\arcsin(\tanh(3/2)) \approx2.263$
\[\mu(d) \leq 2 + \cos^2(d/2) A(d),
\]
where $$A(d) = \frac{6[\sin(d/2) - \arctanh(\sin(d/2))\cos^2(d/2)]}{\sin^3(d/2) - 3\cos^2(d/2)[\sin(d/2) - \arctanh(\sin(d/2))\cos^2(d/2)]}.$$
Moreover, at the limiting case $d=\pi$ we have $\mu'(\pi)=0$ and $\frac14\leq\mu''(\pi)\leq\frac32$.
\end{lemma}

\begin{proof}
Define the test function $y_T = \tanh(u) - u\sech^2(T)$, which by design satisfies the Neumann initial conditions of (\ref{eq:modSLP}). By direct calculation
\begin{eqnarray*}
\begin{split}
    \int_{-T}^T -y_T.y_T''du &= \int_{-T}^T 2\tanh^2(u)\sech^2(u)du - \sech^2(T)\int_{-T}^T 2u\tanh(u)\sech^2(u)du\\
    &= \frac23[2\tanh^3(T)] - \sech^2(T)[2\tanh(T) - 2T\sech^2(T)]
\end{split}
\end{eqnarray*}

\begin{eqnarray*}
\begin{split}
    \int_{-T}^T y_T^2\sech(u)du &= \int_{-T}^T \left(\tanh^2(u) -2u\tanh(u))\sech^2(T) + u^2\sech^4(T)\right)\sech^2(u)du\\
    &= \int_{-T}^T \tanh^2(u)\sech^2(u)du -\sech^2(T)\int_{-T}^T 2u\tanh(u)\sech^2(u)du \\& +  \sech^4(T)\int_{-T}^T u^2\sech^2(u)du\\
    &= \frac13[2\tanh^3(T)] - \sech^2(T)[2\tanh(T) - 2T\sech^2(T)] \\&+ \sech^4(T) \int_{-T}^T u^2\sech^2(u)du.
\end{split}
\end{eqnarray*}
Hence our test function $y_T$ gives us the inequality
\begin{eqnarray*}
\begin{split}
    \mu(d) &\leq \frac{\int_{-T}^T 2\tanh^2(u)\sech^2(u)du - \sech^2(T)\int_{-T}^T 2u\tanh(u)\sech^2(u)du}{\int_{-T}^T \tanh^2(u)\sech^2(u)du -\sech^2(T)\int_{-T}^T 2u\frac{\tanh(u)}{\cosh^2(u)}du + \sech^4(T)\int_{-T}^T u^2\sech^2(u)du}\\
    &\leq \frac{\frac23[2\tanh^3(T)] - \sech^2(T)[2\tanh(T) - 2T\sech^2(T)]}{\frac13[2\tanh^3(T)] - \sech^2(T)[2\tanh(T) - 2T\sech^2(T)]}
\end{split}
\end{eqnarray*}
where the last inequality is valid if $T\geq\frac32$, so for $d\geq2\arcsin(\tanh(3/2))\approx 2.263$.

Using that $\sech^2(T)=1-\tanh^2(T) = 1-\sin^2(d/2)=\cos^2(d/2)$ we have the desired inequality: 
\begin{eqnarray*}
\begin{split}
    \mu(d) &\leq \frac{2\sin^3(d/2) - 3\cos^2(d/2)[\sin(d/2) - \arctanh(\sin(d/2))\cos^2(d/2)]}{\sin^3(d/2) - 3\cos^2(d/2)[\sin(d/2) - \arctanh(\sin(d/2))\cos^2(d/2)]}\\
    &= 2 + \cos^2(d/2) \frac{6[\sin(d/2) - \arctanh(\sin(d/2))\cos^2(d/2)]}{\sin^3(d/2) - 3\cos^2(d/2)[\sin(d/2) - \arctanh(\sin(d/2))\cos^2(d/2)]}
\end{split}
\end{eqnarray*}

Turning our attention to the limiting case of $d=\pi$, we Taylor expand cosine at $\pi/2$ to yield
\begin{equation*}\label{eq:upper2deriv}
    \mu(d) \leq 2 + \frac32(\pi-d)^2 + O((\pi-d)^3).
\end{equation*}

Similarly, we combine Lemma \ref{lem:1sttaylorlambda} and the Taylor expansion of sine at $\pi/2$ to yield

\begin{equation*}\label{eq:lower2deriv}
    2 + \frac14(\pi-d)^2 + O((\pi-d)^3) \leq \mu(d).
\end{equation*}

From this pair of Taylor expansions we calculate that

\begin{equation*}
    \mu'(\pi)=0,\quad\frac14 \leq \mu''(\pi) \leq \frac32 .
\end{equation*}
\end{proof}

\end{document}